\documentclass{amsart}
\usepackage{amscd,amssymb,stmaryrd,url}
\usepackage[all]{xy}

\title{Hochschild homology and trivial extensions}

\author{Petter Andreas Bergh}
\address{Department of Mathematical Sciences, NTNU, NO-7491 Trondheim, Norway}
\email{bergh@math.ntnu.no}

\author{Dag Oskar Madsen}
\address{Faculty of Professional Studies, University of Nordland, NO-8049 Bod{\o}, Norway}
\email{Dag.Oskar.Madsen@uin.no}

\newtheorem{lem}{Lemma}[section]
\newtheorem{prop}[lem]{Proposition}
\newtheorem{cor}[lem]{Corollary}
\newtheorem{thm}[lem]{Theorem}

\theoremstyle{definition}

\newtheorem{example}[lem]{Example}
\newtheorem{remark}[lem]{Remark}

\newcommand{\Ae}{A^{\mathrm e}}

\newcommand{\wQ}{\widetilde Q}
\renewcommand{\r}{\mathfrak r}
\newcommand{\Z}{\mathbb Z}

\newcommand{\hhdim}{\operatorname{HHdim}}
\newcommand{\Hom}{\operatorname{Hom}}

\newcommand{\HH}{\operatorname{HH}}
\renewcommand{\ker}{\operatorname{ker}}

\newcommand{\op}{\operatorname{op}}

\newcommand{\rad}{\operatorname{rad}}
\newcommand{\soc}{\operatorname{soc}}
\renewcommand{\top}{\operatorname{top}}

\begin{document}

\begin{abstract}
We prove that if an algebra is either selfinjective, local or graded, then the Hochschild homology dimension of its trivial extension is infinite.
\end{abstract}



\maketitle

\section{Introduction}

Let $A$ be a finite dimensional algebra over an algebraically closed field. It is well known that if its global dimension is finite, then its Hochschild cohomology and homology groups $\HH_n(A)$ and $\HH^n(A)$ vanish for all sufficiently large $n$. In his classical paper \cite{Ques} on the cohomology of path algebras, Happel remarked that the converse was not known for Hochschild cohomology. As shown in \cite{AvramovIyengar}, it does hold for commutative algebras. However, noncommutative counterexamples were given in \cite{Bgms}: there exist algebras of infinite global dimension for which the Hochschild cohomology groups vanish in high degrees.

The Hochschild \emph{homology} groups of the algebras studied in \cite{Bgms} do not behave as the cohomology groups; they do not vanish in high degrees. This led Han to conjecture in \cite{Hdim} that if the Hochschild homology groups of an algebra vanish in high degrees, then the algebra is necessarily of finite global dimension. Han proved that this holds for monomial algebras, and just as for cohomology it also holds for commutative algebras, by \cite{AvramovPoirrier}. In the subsequent papers \cite{Trun, Glob, Split, Svp}, the conjecture has been shown to hold for several classes of algebras, including Koszul algebras, cellular algebras and local graded algebras.

In this paper, we study the Hochschild homology groups of trivial extensions of algebras. The trivial extension of any algebra is a non-semisimple symmetric algebra, and therefore it always has infinite global dimension. For symmetric algebras, the vector space dimensions of the Hochschild homology groups equal the dimensions of the cohomology groups. Thus, for trivial extension algebras, Han's conjecture states that the Hochschild homology \emph{and} cohomology groups do not all vanish in high degrees. We prove that this holds for trivial extensions of selfinjective algebras, local algebras and graded algebras.

\section{Trivial extension algebras}

Throughout this paper, let $\Bbbk$ be an algebraically closed field, and $A$ a finite dimensional $\Bbbk$-algebra. Any such algebra is Morita equivalent to a basic algebra, and these again are isomorphic to quotients of path algebras by admissible ideals. Thus we may without loss of generality assume that our algebra $A$ is of the form $\Bbbk Q/I$ for some finite quiver $Q$ and admissible ideal $I \subseteq \Bbbk Q$.

Denote by $DA$ the $\Bbbk$-dual $\Hom_{\Bbbk}(A, \Bbbk )$ of $A$, considered as a bimodule. The \emph{trivial extension} of $A$ by $DA$, denoted $T(A)=A \ltimes DA$, is the algebra with underlying vector space $A \oplus DA$, and multiplication given by
$$(a,f) \cdot (b,g)=(ab,ag+fb)$$ for all $a,b \in A$ and $f,g \in DA$. For any finite dimensional algebra, the trivial extension is symmetric. Moreover, there is a close relationship between the quiver $\wQ$ of $T(A)$ and the defining quiver $Q$ of $A$; we recall here this relationship as described in \cite{Pres}.

The radical of $T(A)$ is $\r \oplus DA$, where $\r$ is the radical of $A$. Consequently, the quotient $T(A)/ \rad T(A)$ is isomorphic to $A/\r$, from which it follows that the sizes of the complete sets of primitive orthogonal idempotents of $A$ and $T(A)$ are the same. This means $Q$ and $\wQ$ have the same number of vertices.

Next, consider the square of the radical of $T(A)$. It is given by
$$\rad^2 T(A)=\r^2 \oplus (\mathfrak r DA + DA \mathfrak r),$$
and so
$$\rad T(A)/\rad^2 T(A)=\r/\r^2 \oplus DA/(\mathfrak r DA + DA \mathfrak r).$$
The quotient $DA/(\mathfrak r DA + DA \mathfrak r)$ is isomorphic to $D(\soc_{\Ae} A)$, where $\Ae$ denotes the enveloping algebra $A \otimes_{\Bbbk} A^{\op}$ of $A$, and $\soc_{\Ae} A$ denotes the socle of the $A$-bimodule $A$. Consequently, there is an isomorphism
$$\rad T(A)/\rad^2 T(A) \simeq \r/\r^2 \oplus D(\soc_{\Ae} A)$$
of $A$-bimodules. The presence of the summand $\r/\r^2$ shows that the arrows of $Q$ can be regarded as arrows also of $\widetilde Q$, whereas the summand $D(\soc_{\Ae} A)$ represents additional arrows. The number of such additional arrows in $\wQ$ from a vertex $i$ to a vertex $j$ is equal to $\dim_\Bbbk e_j (D(\soc_{\Ae} A)) e_i$, where $e_i$ denotes the primitive idempotent in $A$ corresponding to the vertex $i$ in the quiver $Q$. Let $\wQ^+$ denote the collection of all new arrows; this set is non-empty since $D(\soc_{\Ae} A) \neq 0$. If $Q_1$ denotes the set of arrows of $Q$ and $\wQ_1$ denotes the set of arrows of $\wQ$, then $\wQ_1=Q_1 \cup \wQ^+$.

Denote by $\xi$ the surjection $DA \rightarrow D(\soc_{\Ae} A)$. For every arrow $\beta \in \wQ^+$ from $i$ to $j$, choose an element $x_\beta \in e_j (DA) e_i$ such that the set
$$\{ \xi(x_\beta) \mid \beta \in \wQ^+, \beta \colon i \rightarrow j \}$$
forms a $\Bbbk$-basis for $e_j (D(\soc_{\Ae} A)) e_i$. Now define a surjective ring homomorphism $\phi \colon \Bbbk \wQ \rightarrow T(A)$ as follows. For a primitive idempotent $e_i$ we set $\phi(e_i)=(e_i,0)$, while $\phi(\alpha)=(\alpha,0)$ for all $\alpha \in Q_1$. For all $\beta \in \wQ^+$ we set $\phi(\beta)=(0,x_\beta)$. Then the kernel $\widetilde I$ of $\phi$ is an admissible ideal in $\Bbbk \wQ$, and $T(A) \simeq \Bbbk \wQ/ \widetilde I$. While the quiver $\wQ$ is an invariant of the algebra $T(A)$, the ideal $\widetilde I$ may depend on the choices made; the algebra might admit several presentations as a bounded path algebra. Nevertheless, the following always holds.

\begin{lem}\label{comp}
Suppose $\beta_1$ and $\beta_2$ are arrows in $\wQ^+$. Then $\beta_2 \beta_1 \in \widetilde I$.
\end{lem}

\begin{proof}
Using the ring homomorphism $\phi \colon \Bbbk \wQ \rightarrow T(A)$, we obtain
$$\phi(\beta_2 \beta_1)=\phi(\beta_2) \cdot \phi(\beta_1)=(0,x_{\beta_2}) \cdot (0,x_{\beta_1})=0,$$
where $x_{\beta_1}, x_{\beta_2} \in DA$. Hence $\beta_2 \beta_1 \in \ker \phi = \widetilde I$.
\end{proof}

\section{Hochschild homology dimension}

In this main section we prove that the Hochschild homology groups of the trivial extension algebra $T(A)$ of $A$ do not all vanish in high degrees, provided the algebra $A$ is either local, selfinjective or graded. We define the \emph{Hochschild homology dimension} of an algebra $B$ as
$$\hhdim B \stackrel{\text{def}}{=} \sup \{ n \mid \HH_n(B) \neq 0 \},$$
and so what we shall prove is that $\hhdim T(A) = \infty$ when $A$ is either local, selfinjective or graded.

We recall first a key result from \cite{Trun}, which shows that the existence of a certain kind of cycle in the quiver of an algebra forces the Hochschild homology dimension to be unbounded. Let $\Gamma$ be a finite quiver, $J$ an admissible ideal in $\Bbbk \Gamma$, and consider the algebra $\Bbbk \Gamma / J$. A \emph{cycle} in $\Gamma$ is a path $p$ of length at least one, starting and ending at the same vertex. Let $p=\alpha_n \dots \alpha_2 \alpha_1$ be a cycle in $\Gamma$ with $\alpha_i \in \Gamma_1$ for all $1 \leq i \leq n$. We say that $p$ is a \emph{$2$-truncated cycle} in the algebra $\Bbbk \Gamma /J$ if $\alpha_{i+1} \alpha_i \in J$ for all $1 \leq i \leq n-1$ and also $\alpha_1 \alpha_n \in J$.

\begin{thm}[{\cite[Theorem 3.1]{Trun}}]\label{bhm}
Let $\Gamma$ be a finite quiver and $J$ an admissible ideal in $\Bbbk \Gamma$. If $\Bbbk \Gamma/J$ admits a $2$-truncated cycle, then $\hhdim \Bbbk \Gamma/J = \infty$.
\end{thm}

In \cite{Trun}, this result was actually proved in a more general setting. Namely, the coefficients need not be taken from a field, they can be taken from any commutative ring. Moreover, the quiver need not be finite, it only has to have a finite number of vertices.

We are now ready to prove our three results, proving that $\hhdim T(A) = \infty$ whenever $A$ is local, selfinjective or graded. We divide these three cases into separate subsections. Recall that we may without loss of generality assume that $A$ is the quotient $\Bbbk Q/I$ of a path algebra, where $Q$ is a finite quiver an $I \subseteq \Bbbk Q$ an admissible ideal. We keep the notation from the previous section, in particular $\wQ$ denotes the quiver of $T(A)$.

\subsection{The local case} When the algebra $A$ is local, the proof is short since the underlying quiver contains only a single vertex.

\begin{thm}
If $A$ is a local finite dimensional $\Bbbk$-algebra, then $\hhdim T(A)= \infty$.
\end{thm}

\begin{proof}
The quiver $Q$, and therefore also $\wQ$, has only one vertex. Let $\beta$ be an arrow in $\wQ^+$. By Lemma \ref{comp}, the path $p=\beta$ is a $2$-truncated cycle in $\Bbbk \wQ/ \widetilde I$, and this algebra is isomorphic to $T(A)$. Hence $\hhdim T(A)= \infty$ by Theorem \ref{bhm}.
\end{proof}

\subsection{The selfinjective case} Next, we treat the case when the algebra $A$ is selfinjective. Before the proof, we establish the following result, showing that every vertex in the quiver $\wQ$ of $T(A)$ is the target of at least one new arrow. That is, given any vertex, there exists one arrow in $\wQ$ which does not correspond to an arrow in the quiver $Q$ of $A$, and having the vertex as target. We denote by $r$ the number of vertices in the quiver $Q$.

\begin{prop}\label{self}
Suppose that $A$ is a selfinjective $\Bbbk$-algebra. Then for every $1 \leq i \leq r$, there is an arrow in $\wQ^+$ ending at vertex $i$.
\end{prop}

\begin{proof}
Since $A$ is selfinjective, there exists a permutation
$$\pi \colon \{1, \dots, r\} \rightarrow \{1, \dots, r\}$$
of the vertices such that $Ae_i \simeq D(e_{\pi(i)}A)$ for all $1 \leq i \leq r$. So for any given $i$, the Loewy length of the left $A$-module $Ae_i$ is equal to the Loewy length of the right $A$-module $e_{\pi(i)}A$. Denote this common length by $L_i$, and choose a nonzero element $x \in \r^{L_i -1} e_i$. Now $\r^{L_i -1} e_i$ equals $\soc_A Ae_i$, which in turn is isomorphic to $D(\top_{A^{\op}} e_{\pi(i)}A)$, where $\top_{A^{\op}} e_{\pi(i)}A$ denotes the top of the right $A$-module $e_{\pi(i)}A$.
Therefore $x=e_{\pi(i)} x e_i$, giving
$$x \in e_{\pi(i)} \r^{L_i -1}= \soc_{A^{\op}} e_{\pi(i)}A.$$

Let $\alpha$ be any arrow in $Q_1$. Then $\alpha x=0$ since $x$ belongs to the socle of the left $A$-module $Ae_i$, and $x \alpha =0$ since $x$ belongs to the socle of the right $A$-module $e_{\pi(i)}A$. Thus $x$ is an element of $\soc_{\Ae} A$, showing that $e_{\pi(i)} (\soc_{\Ae} A) e_i$ is nonzero. But then $e_i \left ( D(\soc_{\Ae} A) \right ) e_{\pi(i)}$ must also be nonzero, and so there exists an arrow in $\wQ^+$ from vertex $\pi(i)$ to $i$.
\end{proof}

\begin{remark}
As can be seen from the proof of this proposition, the result does not require the algebra $A$ to be selfinjective. Namely, the result holds under the weaker assumption that $\soc_A A \subseteq \soc_{\Ae} A$, that is, when the socle of $A$ as a left module is contained in its bimodule socle.
\end{remark}

Proposition \ref{self} guarantees that the quiver of the trivial extension of a selfinjective algebra contains ``enough'' arrows. This is what we need in order to prove the result on the Hochschild homology dimension for such algebras.

\begin{thm}
If $A$ is a selfinjective $\Bbbk$-algebra, then $\hhdim T(A)= \infty$.
\end{thm}

\begin{proof}
From Proposition \ref{self}, we know that for any vertex in $\wQ$ there is at least one arrow in $\wQ^+$ ending at that vertex. It follows from this that there exists a cycle $p = \beta_t \dots \beta_2 \beta_1$ in $\wQ$ consisting entirely of arrows in $\wQ^+$. By Lemma \ref{comp}, the composition of any two such arrows is zero in $T(A)$, hence $p$ is a $2$-truncated cycle in the trivial extension algebra. Theorem \ref{bhm} now gives $\hhdim T(A)= \infty$.
\end{proof}

Since the trivial extension of any algebra is symmetric, we obtain the following.

\begin{cor}
If $A$ is any finite dimensional $\Bbbk$-algebra, then $\hhdim T \left ( T(A) \right ) = \infty$.
\end{cor}

\subsection{The graded case} In this final subsection, we treat the case when the algebra $A$ is positively graded, so that $A=A_0 \oplus A_1 \oplus \dots \oplus A_s$. Many of the algebras one normally studies are gradable, and they are therefore covered by the result. However, our proof requires the characteristic of the ground field $\Bbbk$ to be zero, and the degree zero part $A_0$ of $A$ to be isomorphic to a product $\Bbbk \times \dots \times \Bbbk=\Bbbk^{\times r}$ as a $\Bbbk$-algebra. Let $1_A = e_1 + \dots + e_r$ be the corresponding decomposition of the identity.

\begin{example}
We do not require the generators of $A$ to be in degree $1$. For example, let $A$ be the path algebra $A=\Bbbk Q/I$, where $Q$ is the quiver
$$Q \colon
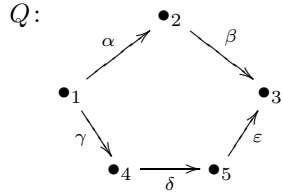
\xymatrix@C=0.8ex@R=4ex{&&\bullet_2 \ar[drr]^{\beta} \\ \bullet_1 \ar[urr]^{\alpha} \ar[dr]_{\gamma} &&  && \bullet_3\\ & \bullet_4 \ar[rr]_{\delta} && \bullet_5 \ar[ur]_{\varepsilon}}$$
and $I$ is the ideal in $\Bbbk Q$ generated by the single relation $\varepsilon \delta \gamma - \beta \alpha$. This ideal is not homogeneous with the path length grading, where every arrow is assigned the degree one. However, if we let $\deg \alpha=\deg \beta=3$ and $\deg \gamma=\deg \delta=\deg \varepsilon=2$, then $A$ is a positively graded algebra.

There are finite dimensional algebras that do not admit a positive grading with semisimple degree zero part, see \cite{Ungr} for examples.
\end{example}

Returning now to the general case, for $1 \leq l \leq s$, let $C^l_A$ be the $r \times r$-matrix with entries $(C^l_A)_{i,j}= \dim_{\Bbbk} e_j A_l e_i$. The \emph{graded Cartan matrix of $A$} is defined to be the $r \times r$-matrix
$$C_A(x) = C^0_A + C^1_A x + C^2_A x^2 + \dots + C^s_A x^s$$
with entries in $\Z[x]$. Its determinant $\det {C_A(x)}$ is the \emph{graded Cartan determinant of $A$}. The following result from \cite{Glob} establishes a connection between this determinant and the Hochschild homology dimension of $A$.

\begin{thm}[{\cite[Corollary 3.5]{Glob}}]\label{grad}
If $\det {C_A(x)} \neq 1$, then $\hhdim A= \infty$.
\end{thm}

Now we show how to give the trivial extension $T(A)$ of $A$ a positive grading, based on the grading of $A$. By definition, there is a vector space decomposition $T(A) = A \oplus DA$. We keep the original grading of $A$, so that
$$\deg A_l=l; \quad 1 \leq l \leq s,$$
and then we give $D(A)$ the following grading:
$$\deg D(A_l)=s+1-l; \quad 1 \leq l \leq s.$$
In this way, $T(A)$ becomes a positively graded $\Bbbk$-algebra with top degree $s+1$.

Next, we analyze the graded Cartan matrix
$$C_{T(A)}(x) = C^0_{T(A)} + C^1_{T(A)} x + C^2_{T(A)} x^2 + \dots + C^s_{T(A)}  x^s + C^{s+1}_{T(A)}  x^{s+1}$$
of $T(A)$. The matrices $C^0_{T(A)}$ and $C^{s+1}_{T(A)}$ must both be identity matrices, since
$$\dim_{\Bbbk} e_j (T(A)_0) e_i=\dim_{\Bbbk} e_j A_0 e_i= \delta_{ij}$$
and
$$\dim_{\Bbbk} e_j (T(A)_{s+1}) e_i=\dim_{\Bbbk} e_j D(A_0) e_i= \delta_{ij},$$
where $\delta_{ij}$ denotes the Kronecker delta. It follows from this that the graded Cartan matrix of $T(A)$ has the shape
$$C_{T(A)}(x)=
\begin{pmatrix}
1 + p_{1,1}(x) + x^{s+1} & p_{1,2}(x) & \cdots & p_{1,r}(x) \\
p_{2,1}(x) & 1 + p_{2,2}(x) + x^{s+1} & \cdots & p_{2,r}(x) \\
\vdots & \vdots & \ddots & \vdots \\
p_{r,1}(x) & p_{r,2}(x) & \cdots & 1 + p_{r,r}(x) + x^{s+1}
\end{pmatrix},
$$
where the entries $p_{i,j}(x)$, $1 \leq i,j \leq r$, are integer polynomials of degree at most $s$ and with zero constant term. This is the key ingredient when we now prove that the Hochschild homology dimension of $T(A)$ is infinite.

\begin{thm}
Suppose that the characteristic of $\Bbbk$ is zero, and let $A=A_0 \oplus A_1 \oplus \dots \oplus A_s$ be a positively graded finite dimensional $\Bbbk$-algebra, where $A_0$ is a product of copies of $\Bbbk$. Then $\hhdim T(A)= \infty$.
\end{thm}

\begin{proof}
The product of the diagonal entries in the graded Cartan matrix $C_{T(A)}(x)$ is a monic polynomial of degree $r(s+1)$, and with constant term $1$. All other products in the expression for the determinant involve off-diagonal entries, so they have degrees less than $r(s+1)$ and zero constant term. Therefore the determinant is of the form $$\det C_{T(A)}(x)= 1 + \dots + x^{r(s+1)}.$$
Then $\hhdim T(A)= \infty$ by Theorem $\ref{grad}$.
\end{proof}

\end{document}